\documentclass[12pt]{article}
\usepackage{float}
\usepackage{amsmath,amssymb,amsthm}
\usepackage{mathtools}
\usepackage{graphicx}
\usepackage{booktabs}
\usepackage{natbib}
\usepackage{hyperref}
\usepackage[margin=1in]{geometry}
\usepackage{enumitem}
\usepackage{xcolor}
\usepackage{subcaption}
\DeclareMathOperator*{\argmax}{arg\,max}

\newtheorem{theorem}{Theorem}[section]

\newtheorem{proposition}[theorem]{Proposition}
\newtheorem{corollary}[theorem]{Corollary}
\newtheorem{definition}[theorem]{Definition}
\newtheorem{remark}[theorem]{Remark}

\newcommand{\E}{\mathbb{E}}
\newcommand{\Var}{\mathrm{Var}}
\newcommand{\R}{\mathbb{R}}

\newcommand{\ind}{\mathbf{1}}
\newcommand{\cF}{\mathcal{F}}
\newcommand{\Cov}{\text{Cov}}

\title{A Selection Premium Decomposition for the Expected Maximum of Random Walks}

\author{
Victor H. de la Pena$^{1}$\thanks{Email: \texttt{vhd1@columbia.edu}} \and Fangyuan Lin$^{1}$\thanks{Email: \texttt{fl2744@columbia.edu}} \and Victor Keegan de la Pe\~na$^{1,2}$\thanks{Email: \texttt{vkd2107@columbia.edu}} \\[0.3cm]
\small $^{1}$\textit{Department of Statistics, Columbia University }\\
\small $^{2}$Infremacy, LLC 
}

\date{February 2026}

\begin{document}

\maketitle

\begin{abstract}
When $K$ models are evaluated on the same validation set of size $n$, the selected winner's apparent performance is biased upward. Suppose $K$ models are evaluated on a shared sequence of i.i.d. observations $X_1,\dots, X_n$, where model $k$ achieves response $f_k(X_i)$ with mean $\mu_k = \mathbb E[f_k(X)]$. Writing $Y_{i,k} = f_k(X_i)-\mu_k$ for the centered increment and $S_{n,k} = \sum_{i=1}^n Y_{i,k}$ for the centered cumulative score, the expected maximum satisfies
\[
0\le\E\bigl[\max_k S_{n,k}\bigr] = \sum_{i=1}^n \E\bigl[\varphi_K(S_{i-1})\bigr]
\]
where $\varphi_K(u) = \E\bigl[\max_k(u_k + Y_k)\bigr] - \max_k u_k$, $u\in \mathbb R^K$, is the selection premium function. This formula corresponds to the “null hypothesis” case (all models are equal in the sense that they have the same mean), which clarifies that the bias arises from selection. While this decomposition follows from elementary conditioning and telescoping, we develop the analytical consequences in five directions. (i)~structural properties of $\varphi_K$; (ii)~extension to stopping times, recovering Wald's equation at $K=1$; (iii)~a winner's curse decomposition for heterogeneous means; (iv)~a universal bias concentration law showing that the first $\alpha$-fraction of observations generates a $\sqrt\alpha$-fraction of total bias.
\medskip

\medskip

\end{abstract}

\section{Introduction}

Model selection is unavoidable in modern statistics and machine learning: a practitioner trains $K$ candidate models and reports the performance of the best. When all candidates are evaluated on the same dataset, the winning model's apparent performance is biased upward---a manifestation of the general phenomenon variously called the winner's curse \cite{thaler1988winner} or selection bias. Understanding this bias is crucial for simulation optimization \cite{nelson1995using, chick2001new}, clinical trials \citep{tripepi2010selection}, backtesting trading strategies, and the ubiquitous practice of model selection in machine learning \citep{cawley2010over}.

Wald's equation \cite{wald1944cumulative} is a fundamental result connecting expectations of random sums to expectations of their summands. In its simplest form, if $Y_1, Y_2, \ldots$ are i.i.d.\ with $\E[Y_i] = 0$ and $T$ is a stopping time with $\E[T] < \infty$, then $\E\biggl[\sum_{i=1}^T Y_i\biggr] = \E[T] \cdot \E[Y_1] = 0.$ This result concerns a single sum. In many applications, we observe not a single sum but $K$ parallel sums and select the maximum. Consider $K$ treatments evaluated on a common sequence of patients, or $K$ machine learning models evaluated on a shared validation set. A natural question emerges: what is $\mathbb E[\max_k S_{n,k}]$ when $S_{n,k} = \sum_{i=1}^n Y_{i,k}$ are $K$ centered random walks driven by a common sequence of observations? This quantity, always non-negative by Jensen's inequality, measures the pure selection bias---the expected inflation from choosing the best among $K$ random trajectories even when all have mean zero. 

This paper provides an exact answer. Theorem~\ref{thm:main} decomposes the expected maximum into a sum of per-step \emph{selection premiums}:
\begin{equation}
\E\bigl[\max_k S_{n,k}\bigr] = \sum_{i=1}^n \E\bigl[\varphi_K(S_{i-1})\bigr],
\end{equation}
where selection premium function $\varphi_K(u) = \E[\max_k(u_k + Y_k)] - \max_k u_k$ quantifies the expected gain from selecting the best arm \emph{after} observing one more data point, relative to committing to the current leader. We then establish structural properties of $\varphi_K$: non-negativity, translation invariance, global maximum at the tie configuration, and vanishing when a leader is decisive.

It extends to random stopping times under the same integrability conditions as classical Wald, and hence $K = 1$, the identity reduces to Wald's equation. For the heterogeneous case where we consider models with distinct means $\mu_1, \ldots, \mu_K$, the identity decomposes the winner's curse into a drift contribution and a selection premium (Theorem~\ref{thm:winners_curse}).

Perhaps the most revealing consequence is the precise characterization of how selection bias accumulates over time. For Gaussian increments, we prove (Theorem~\ref{thm:gaussian_decay}) that the per-step premium satisfies the decay
\[
\E\bigl[\varphi_K(S_{i-1})\bigr] = \sigma\,g(K)\bigl(\sqrt{i} - \sqrt{i-1}\bigr) \sim \frac{\sigma\,g(K)}{2\sqrt{i}},
\]
where $g(K) = \E[\max_{k \le K} Z_k]$ for i.i.d.\ standard normals.

Under finite second moments, a multivariate CLT implies $\E[\max_k S_{n,k}]=c\sqrt{n}+o(\sqrt{n})$ (Theorem~\ref{thm:general_qualitative}). Combining this with the telescoping form $\E[\varphi_K(S_{i-1})]=\E[M_i]-\E[M_{i-1}]$ yields a universal \emph{bias concentration} law: the first $\alpha$-fraction of steps contributes asymptotically a $\sqrt{\alpha}$-fraction of the total bias (Theorem~\ref{thm:general_qualitative}(b)). Therefore, selection bias is primarily a small-sample phenomenon: it accumulates during the early ``competition phase'' when models are close together, and new observations contribute negligible additional bias once a leader has emerged. For Gaussian increments, the premium drops to $10\%$ of its initial value by step $\lceil 1/(4\cdot 0.01)\rceil = 26$.


\section{Problem Formulation}\label{sec:formulation}

Let $(\Omega, \cF, P)$ be a probability space. We observe a sequence $X_1, X_2, \ldots$ of i.i.d.\ random elements taking values in a measurable space $\mathcal{X}$. For each arm $k \in \{1, \ldots, K\}$, define:
\begin{itemize}
\item $f_k: \mathcal{X} \to \mathbb R$ --- measurable response function
\item $\mu_k = \E[f_k(X)]$ --- true mean effect
\item $Y_{i,k} = f_k(X_i) - \mu_k$ --- centered increment
\item $S_{n,k} = \sum_{i=1}^n Y_{i,k}$ --- centered cumulative sum

\end{itemize}

The true best model is $k^* \in \argmax_{k\in[K]}\mu_k$. When the argmax is unique, we say there is a unique best model with gap $\Delta = \mu_{k^*} - \max_{j\ne k^*}\mu_k > 0$.

The vector $Y_i = (Y_{i,1}, \ldots, Y_{i,K})\in \mathbb R^K$ is determined by the single observation $X_i$, so increments across arms are dependent within each time step but independent across time:
\begin{equation}\label{eq:independence}
Y_i \perp Y_j \quad \text{for } i \neq j.
\end{equation}

Let $\cF_n = \sigma(X_1, \ldots, X_n)$ denote the natural filtration, with $\cF_0 = \{\emptyset, \Omega\}$. Define:
\begin{equation}\label{eq:Mn_def}
M_n := \max_k S_{n,k}
\end{equation}
the maximum cumulative sum across all $K$ models. Our central object of interest is $\E[M_n]$.

\subsection{Standing assumptions}
For the purpose of showing the main identity, we assume:
\begin{enumerate}[label=(A\arabic*)]
    \item \textbf{(i.i.d. increments)} $Y_1, Y_2,\dots$ are i.i.d. in $\mathbb R^K$ and $Y_i$ is independent of $\mathcal F_{i-1}.$
    \item \textbf{(centering)} $\mathbb E[Y_1] = 0$ (i.e. $\mathbb E[Y_{1,k}] = 0$ for each $k$.)
    \item \textbf{(integrability)} $\mathbb E[\max_{1\le k\le K} |Y_{1,k}|] < \infty$.
\end{enumerate}

Assumption (A2) is used for the main identity; we relax it in Section~\ref{sec:unequal_means} when discussing unequal means. Assumption (A3) holds if $\mathbb E|Y_{1,k}|<\infty$ for all $k$, since the number of models $K < \infty$. 

\begin{definition}[Selection Premium Function]\label{def:selection_premium}
For $u = (u_1, \ldots, u_K) \in \R^K$, the \emph{selection premium} (or \emph{selection inflation function}) is:
\begin{equation}\label{eq:phi_def}
\varphi_K(u) = \E\bigl[\max_k(u_k + Y_k)\bigr] - \max_k u_k.
\end{equation}
This function quantifies the expected gain from random perturbations compared to selecting from the current state $u$.

\begin{remark}[Interpretation of $\varphi_K(u)$]\label{rmk:selection-premium-interpretation}
    Imagine standing at position $u = (u_1,\dots, u_K)$ which represents current cumulative performance of $K$ models. After one more shared observation, the position of the $k$th model becomes $u_k+Y_k$ where $Y_k$ is random.
    \begin{itemize}
        \item If you must commit to the current best model, you get $\max_k u_k$.
        \item If you can wait and pick the best after observing $(Y_1,\dots, Y_K)$, you expect $\E[\max_k (u_k + Y_k)]$.
    \end{itemize}
    The selection premium function $\varphi_K$ is the value of the option to switch: the expected gain from being able to select ex-post rather than ex-ante. This ``option value'' is always non-negative (Proposition~\ref{prop:phi_properties}) and depends on how close the competitors are. When the $K$ models are tied, the option is most valuable and $\varphi_K(\mathbf{0}) = \E[\max_k Y_k]$. When one model dominates, the option has less value (Proposition~\ref{prop:phi_properties}).
\end{remark}
\end{definition}

\section{The Per-Step Decomposition}\label{sec:main_results}


Define $M_n := \max_k S_{n,k}$.

\begin{theorem}[Decomposition identity]\label{thm:main}
Under (A1)-(A3). 

\begin{equation}\label{eq:main_identity}
\E[M_n] = \E\bigl[\max_k S_{n,k}\bigr] = \sum_{i=1}^n \E\bigl[\varphi_K(S_{i-1})\bigr]
\end{equation}
where $S_0 = \mathbf{0} = (0, \ldots, 0)\in \mathbb R^K$.
\end{theorem}

\begin{proof}
By the independence $Y_i \perp \mathcal F_{i-1}$ and measurability of $S_{i-1}$ with respect to $\mathcal F_{i-1}$,
\begin{equation}
    \E[M_i \mid \mathcal F_{i-1}] = \E[\max_k(S_{i-1,k}+Y_{i,k})\mid\mathcal F_{i-1}] = g(S_{i-1}).
\end{equation}
where $g(u) = \E[\max_k(u_k+Y_k)]$. Therefore,
\begin{equation}
    E[M_{i}] - E[M_{i-1}] = E[g(S_{i-1})] - E[\max_k S_{i-1,k}] = E[\varphi_K(S_{i-1})].
\end{equation}
Summing from $i=1$ to $n$ and using $M_0 = 0$,
\begin{equation}
    E[M_n] = \sum_{i=1}^n (E[M_i] - E[M_{i-1}]) = \sum_{i=1}^n E[\varphi_K(S_{i-1})]
\end{equation}
\end{proof}

\begin{remark}[On the proof]
The identity is a telescoping sum with each increment computed by conditioning. Consider the Doob decomposition of the submartingale 
\begin{equation}
    M_n = M_0 + A_n + N_n
\end{equation}
where $A_n$ is a predictable increasing process with $A_0 = 0$ and $N_n$ is a martingale with $N_0 = 0$. It produces the same result ($\Delta A_i = \varphi_K(S_{i-1})$ is exactly the predictable compensator),
but the direct argument above makes the elementary nature transparent.
The value of the identity lies not in its proof but in the analytical and interpretational consequences developed in the subsequent sections.
\end{remark}

\subsection{Properties of the Selection Premium Function}

\begin{proposition}[Properties of $\varphi_K$]\label{prop:phi_properties}
The selection premium function $\varphi_K$ satisfies:
\begin{enumerate}[label=(\roman*)]
\item Non-negativity: $\varphi_K(u) \geq 0$ for all $u \in \R^K$.
\item No selection, no problem: $\varphi_1(u) = 0$ for all $u \in \R$.
\item Translation invariance: For any $u \in \R^K$ and any constant $c \in \R$,
\begin{equation}
\varphi_K(u + c\mathbf{1}) = \varphi_K(u),
\end{equation}
where $\mathbf{1} = (1,\dots,1) \in \R^K$.
\item Vanish when leader is decisive: If $|Y_k| \leq b$ a.s.\ for all $k$, and $\max_k u_k - \max_{j \neq k^*} u_j > 2b$, then $\varphi_K(u) = 0$.
\item global maximum at ties: $\varphi_K(u) \leq \varphi_K(\mathbf{0})$ for all $u \in \R^K$.
\end{enumerate}
\end{proposition}

\begin{proof}
(i) The max function is convex, so by Jensen's inequality:
\[
\E[\max_k(u_k + Y_k)] \geq \max_k \E[u_k + Y_k] = \max_k u_k.
\]

(ii) When $K = 1$: $\varphi_1(u) = \E[u + Y] - u = \E[Y] = 0$.

(iii) By definition,
\[
\varphi_K(u + c\mathbf{1})
= \E\bigl[\max_k (u_k + c + Y_k)\bigr] - \max_k (u_k + c).
\]
Since $\max_k (u_k + c + Y_k) = c + \max_k (u_k + Y_k)$ and
$\max_k (u_k + c) = c + \max_k u_k$, the constant $c$ cancels:
\[
\varphi_K(u + c\mathbf{1})
= \E[\max_k (u_k + Y_k)] - \max_k u_k
= \varphi_K(u).
\]

(iv) When the leader's gap exceeds $2b$ and $|Y_k| \leq b$ a.s., no noise realization can change the identity of the leader, so $\max_k(u_k + Y_k) = u_{k^*} + Y_{k^*}$ a.s., giving $\varphi_K(u) = \E[Y_{k^*}] = 0$.

(v) Let $m := \max_k u_k$. Then $u_k \leq m$ for all $k$, so:
\begin{align}
    u_k + Y_k &\leq m + Y_k \quad \text{for all } k,\\
    \max_{k}(u_k + Y_k) &\leq \max_k(m + Y_k) = m + \max_k Y_k.
\end{align}
Taking expectations:
\begin{align}
    \E[\max_{k}(u_k + Y_k)] &\leq m + \E[\max_k Y_k],\\
    \varphi_K(u) = \E[\max_{k}(u_k + Y_k)] - m &\leq \E[\max_k Y_k] = \varphi_K(\mathbf{0}).
\end{align}
\end{proof}

\begin{corollary}\label{thm:bounds}
Assume (A1)-(A3). For fixed $n$ and $K \geq 1$:
\begin{equation}\label{eq:bounds}
0 \leq \E\bigl[\max_k S_{n,k}\bigr] \leq n \cdot \E\bigl[\max_k Y_k\bigr].
\end{equation}
\end{corollary}

\begin{proof}
Non-negativity follows from the non-negative of the selection premium function $\varphi$.

For the upper bound, by Proposition~\ref{prop:phi_properties}:
\[
\E\bigl[\max_k S_{n,k}\bigr] = \sum_{i=1}^n \E[\varphi_K(S_{i-1})] \leq \sum_{i=1}^n \varphi_K(\mathbf{0}) = n \cdot \E[\max_k Y_k].
\]
\end{proof}

The upper bound in \eqref{eq:bounds} is typically crude: under mild
conditions (e.g.\ finite second moments) we will show that
$\E[\max_k S_{n,k}] = \Theta(\sqrt{n})$ as $n\to\infty$.

\subsection{Extension to Stopping Times}
The fixed-$n$ identity extends to random stopping times.
\begin{theorem}[Stopping time extension]\label{thm:stopping}
Assume \textnormal{(A1)--(A3)}. Let $T$ be a stopping time with respect to
$(\cF_n)$ satisfying $\E[T] < \infty$.
Then:
\begin{equation}\label{eq:stopping_identity}
\E[M_T] \;=\; \E\bigl[\max_k S_{T,k}\bigr]
\;=\; \sum_{i=1}^{\infty}
      \E\bigl[\varphi_K(S_{i-1})\,\ind\{T \geq i\}\bigr].
\end{equation}
\end{theorem}

\begin{proof}
Let $M_n = M_0 + A_n + N_n$ be the Doob decomposition from
Theorem~\ref{thm:main}, with $A_0 = N_0 = 0$.

Since $\varphi_K \geq 0$ (Proposition~\ref{prop:phi_properties}(i)),
\[
A_{T\wedge n}
= \sum_{i=1}^{T\wedge n} \varphi_K(S_{i-1})
= \sum_{i=1}^{n} \varphi_K(S_{i-1})\,\ind\{T \geq i\}
\;\uparrow\; A_T
\quad\text{a.s.}
\]
By monotone convergence,
\begin{equation}\label{eq:AT}
\E[A_T]
= \sum_{i=1}^{\infty}
  \E\bigl[\varphi_K(S_{i-1})\,\ind\{T \geq i\}\bigr].
\end{equation}
Moreover, using $\varphi_K(u) \leq \varphi_K(\mathbf{0}) = \E[\max_k Y_k]$
(Proposition~\ref{prop:phi_properties}(v)),
\begin{equation}\label{eq:AT_finite}
\E[A_T]
\leq \E[\max_k Y_k]\;\E[T]
< \infty.
\end{equation}

We show that the stopped martingale is dominated by an integrable
random variable, which implies uniform integrability.

For $|M_{T\wedge n}|$, the Lipschitz property of $\max$ gives
$|M_i - M_{i-1}| \leq \max_k |Y_{i,k}|$, so
\[
|M_{T\wedge n}|
\;\leq\; \sum_{i=1}^{T\wedge n} \max_k |Y_{i,k}|
\;\leq\; \sum_{i=1}^{T} \max_k |Y_{i,k}|
\;=:\; Z_1.
\]
By Wald's identity for non-negative summands,
$\E[Z_1] = \E[T]\,\E[\max_k |Y_{1,k}|] < \infty$ under (A3).
Combined with \eqref{eq:AT_finite},
\[
|N_{T\wedge n}|
= |M_{T\wedge n} - A_{T\wedge n}|
\leq |M_{T\wedge n}| + A_{T\wedge n}
\leq Z_1 + A_T
=: Z,
\]
and $\E[Z] \leq \E[T]\bigl(\E[\max_k |Y_{1,k}|]
+ \E[\max_k Y_k]\bigr) < \infty$.
Hence $\{N_{T\wedge n} : n \geq 0\}$ is dominated by the integrable
random variable $Z$ and is therefore uniformly integrable.

By Doob's optional stopping theorem, uniform integrability of
$(N_{T\wedge n})$ gives $\E[N_T] = \E[N_0] = 0$.
Since $M_T = A_T + N_T$,
\[
\E[M_T] = \E[A_T] + \E[N_T] = \E[A_T]
= \sum_{i=1}^{\infty}
  \E\bigl[\varphi_K(S_{i-1})\,\ind\{T \geq i\}\bigr].
  \qedhere
\]
\end{proof}

\begin{proposition}[Reduction to Wald's equation]\label{prop:wald_reduction}
Let $K=1$. Assume \textnormal{(A1)--(A3)} and let $T$ be a stopping time
with $\E[T] < \infty$. Then $\E[S_T] = 0$.
More generally, if $\E[Y_1] = \mu \neq 0$, then $\E[S_T] = \mu\,\E[T]$.
\end{proposition}

\begin{proof}
When $K=1$, the maximum is trivial: $M_n = S_n$.
By Definition~\ref{def:selection_premium},
\[
\varphi_1(u) = \E[u + Y_1] - u = \E[Y_1] = 0
\quad \text{for all } u \in \R.
\]
Applying Theorem~\ref{thm:stopping} with $K=1$ yields
\[
\E[S_T] = \E[M_T]
= \sum_{i=1}^{\infty} \E\bigl[\varphi_1(S_{i-1})\,\ind\{T \ge i\}\bigr]
= 0.
\]
For the general case, write $Y_i = \mu + (Y_i - \mu)$ and apply
the centered result to the increments $\tilde{Y}_i = Y_i - \mu$:
\[
\E\biggl[\sum_{i=1}^T \tilde{Y}_i\biggr] = 0
\quad\Longrightarrow\quad
\E[S_T] = \mu\,\E[T].
\]
\end{proof}

\section{Extension to Unequal Means: The Winner's Curse}\label{sec:unequal_means}

The previous sections analyze the {null (equal-mean)} case, in which
$\E[M_n]=\E[\max_k S_{n,k}]$ is \emph{pure selection bias}.
We now allow heterogeneous means $\mu_1,\dots,\mu_K$. In this setting the
winner's observed score contains a deterministic {drift advantage} plus a
residual {selection premium} coming from random fluctuations.

Recall the centered increments $Y_{i,k}=f_k(X_i)-\mu_k$ and centered walks
$S_{n,k}=\sum_{i=1}^n Y_{i,k}$. Define the (uncentered) cumulative scores
\[
\bar S_{n,k}:=\sum_{i=1}^n f_k(X_i)=n\mu_k+S_{n,k},
\qquad
\hat k_n\in\argmax_{k\in[K]}\bar S_{n,k}.
\]
Let $k^*\in\argmax_k \mu_k$ be a true best model and define the gaps
$\Delta_k:=\mu_{k^*}-\mu_k\ge 0$, with $\Delta_{k^*}=0$.

Consider the {drifted} (relative-to-best) walks
\begin{equation}\label{eq:drifted_walk_def}
R_{n,k}:=S_{n,k}-n\Delta_k,
\qquad
R_n:=(R_{n,1},\dots,R_{n,K}).
\end{equation}
Then
\begin{equation}\label{eq:max_uncentered_decompose}
\max_{k\in[K]}\bar S_{n,k}
=
n\mu_{k^*}+\max_{k\in[K]}R_{n,k}.
\end{equation}
Moreover, $R_n$ evolves by
\[
R_{n,k}=R_{n-1,k}+(Y_{n,k}-\Delta_k),
\]
so the drifted increments are $Y_n-\Delta$.

\begin{definition}[Heterogeneous selection premium]\label{def:phi_hetero}
For $u\in\R^K$, define
\begin{equation}\label{eq:phi_hetero_def_fixed}
\phi_{K,\Delta}(u)
:=
\E\!\left[\max_{k\in[K]}\bigl(u_k+Y_k-\Delta_k\bigr)\right]
-
\max_{k\in[K]}u_k,
\end{equation}
where $Y\stackrel{d}{=}Y_1$ is independent of $u$.
\end{definition}

\begin{remark}[Interpretation: option value in the drifted race]
At drifted state $u$, committing now yields $\max_k u_k$.
Waiting one more observation and then selecting yields
$\E[\max_k(u_k+Y_k-\Delta_k)]$.
Thus $\phi_{K,\Delta}(u)$ is the one-step \emph{option value of waiting} in the
drifted race. When $\Delta\equiv 0$, it reduces to the homogeneous premium
$\varphi_K(u)$.
\end{remark}
\begin{theorem}[Unequal means]\label{thm:winners_curse}
Under \textnormal{(A1)} and \textnormal{(A3)}, let $k^* \in \argmax_k \mu_k$ be the true best model and $\hat{k}_n := \argmax_k \bar{S}_{n,k}$ the selected model. Then:

\begin{enumerate}[label=(\alph*)]
\item {Winner's curse decomposition:}
With $R_n$ defined by \eqref{eq:drifted_walk_def},
\begin{equation}\label{eq:winners_curse_decomp_fixed}
\E\Big[\max_{k\in[K]}\bar S_{n,k}\Big]
=
n\mu_{k^*}
+
\sum_{i=1}^n \E\Big[\phi_{K,\Delta}\big(R_{i-1}\big)\Big].
\end{equation}
Equivalently,
\begin{equation}\label{eq:drifted_max_identity_fixed}
\E\Big[\max_{k\in[K]}R_{n,k}\Big]
=
\sum_{i=1}^n \E\Big[\phi_{K,\Delta}\big(R_{i-1}\big)\Big].
\end{equation}

\item {Optimism of the selected model:} The observed score overestimates the true mean of the selected model:
\begin{equation}\label{eq:optimism}
\E\biggl[\frac{\bar{S}_{n,\hat{k}_n}}{n} - \mu_{\hat{k}_n}\biggr] = \frac{1}{n}\E\bigl[S_{n,\hat{k}_n}\bigr] \geq 0.
\end{equation}

\item {Bound by the null case:} The heterogeneous selection premium is dominated by the homogeneous case:
\begin{equation}\label{eq:hetero_bound}
0 \leq \E\bigl[\max_k\bigl(R_{n,k}\bigr)\bigr] \leq \E\bigl[\max_k S_{n,k}\bigr].
\end{equation}

\item {Vanishing selection premium:} If $k^*$ is unique, then
\begin{equation}\label{eq:vanishing_premium}
\frac{1}{n}\E\bigl[\max_k\bigl(R_{n,k}\bigr)\bigr] \to 0 \quad \text{as } n \to \infty.
\end{equation}
\end{enumerate}
\end{theorem}

\begin{proof}
{(a)} Similar to \ref{thm:main}, by telescoping.

{(b)} Note that $S_{n,\hat{k}_n} = \max_k \bar{S}_{n,k} - n\mu_{\hat{k}_n}$, and:
\begin{align*}
\E[S_{n,\hat{k}_n}] &= \E[\max_k \bar{S}_{n,k}] - n\E[\mu_{\hat{k}_n}]\\
&\geq \max_k \E[\bar{S}_{n,k}] - n\E[\mu_{\hat{k}_n}] \quad \text{(Jensen)}\\
&= n\mu_{k^*} - n\E[\mu_{\hat{k}_n}] \geq 0.
\end{align*}
The last inequality holds because $\mu_{\hat{k}_n} \leq \mu_{k^*}$ a.s.

{(c)} Since $\Delta_k \geq 0$: $S_{n,k} - n\Delta_k \leq S_{n,k}$ for all $k$, hence $\max_k(S_{n,k} - n\Delta_k) \leq \max_k S_{n,k}$. For non-negativity, let $V_k:=S_{n,k}-n\Delta_k$.
\[
\E\!\left[\max_{k\in[K]} V_k\right]\overset{Jensen}{\ge} \max_{k\in[K]}\E[V_k]
= \max_{k\in[K]}(\E[S_{n,k}]-n\Delta_k)
= \max_{k\in[K]}(-n\Delta_k)=0,
\]
since $\E[S_{n,k}]=0$ and $\min_k\Delta_k=\Delta_{k^*}=0$.

{(d)} By the law of large numbers, $S_{n,k}/n \to 0$ a.s.\ for all $k\in\{1,\dots, K\}$. Hence:
\[
\frac{1}{n}\max_{k\in [K]}(S_{n,k} - n\Delta_k) = \max_{k\in [K]}\biggl(\frac{S_{n,k}}{n} - \Delta_k\biggr) \to \max_{k\in[K]}(-\Delta_k) = 0 \quad \text{a.s.}
\]
\end{proof}

Part~(c) shows that heterogeneity can only {reduce} selection bias: if one model is truly better, the competition phase ends sooner and less selection premium can accrue. Part~(d) further indicates that this excess optimism is {transient} in the sense that the bias per observation vanishes as $n\to\infty$. Since heterogeneity cannot increase selection bias, for the remainder of the paper, we work under the null (equal-mean) case $\mu_1=\cdots=\mu_K$, which also yields cleaner notation.

\section{Decay of the Selection Premium}\label{sec:decay}

The identity $\E[M_n] = \sum_{i=1}^n \E[\varphi_K(S_{i-1})]$ expresses the total selection bias as a sum of per-step premiums. This section establishes how these premiums decay---first exactly for Gaussian increments, then asymptotically for general distributions.

\begin{definition}[Normalized Selection Premium]\label{def:normalized_premium}
Define the \emph{normalized selection premium at step $i$} as
\begin{equation}
\psi_K(i) := \frac{\E[\varphi_K(S_{i-1})]}{\varphi_K(\mathbf{0})}, \quad i = 1, 2, \ldots
\end{equation}
By Proposition~\ref{prop:phi_properties}(v), $\psi_K(i) \in [0, 1]$ for all $i$. Note that $\psi_K(1) = 1$.
\end{definition}

\subsection{Exact decay for Gaussian increments}

\begin{theorem}[Gaussian decay]\label{thm:gaussian_decay}
Let $Y_i = (Y_{i,1},\ldots,Y_{i,K})$ be i.i.d.\ Gaussian vectors with $\E[Y_i] = \mathbf{0}$ and $\mathrm{Cov}(Y_{i,j}, Y_{i,k}) = \Sigma_{jk}$. Then for all $i \geq 1$:
\begin{enumerate}[label=(\alph*)]
\item The expected maximum satisfies:
\begin{equation}\label{eq:gaussian_Mn}
\E[M_n] = \sqrt{n}\,\E\left[\max_{k\in[K]} Z_{k}\right],
\end{equation}
where $Z = (Z_1,\ldots,Z_K) \sim N(\mathbf{0}, \Sigma)$.

\item The per-step selection premium is:
\begin{equation}\label{eq:gaussian_premium_exact}
\E[\varphi_K(S_{i-1})] = \E\left[\max_{k\in[K]} Z_{k}\right]\bigl(\sqrt{i} - \sqrt{i-1}\bigr), \quad i \geq 1.
\end{equation}

\item The normalized selection premium is:
\begin{equation}\label{eq:psi_gaussian}
\psi_K(i) := \frac{\E[\varphi_K(S_{i-1})]}{\varphi_K(\mathbf{0})} = \sqrt{i} - \sqrt{i-1} \sim \frac{1}{2\sqrt{i}}\quad (i \to \infty),
\end{equation}
which is independent of $K$ and $\Sigma$.

\item The map $n \mapsto \E[M_n]$ is concave: for all $n \geq 1$,
\begin{equation}\label{eq:concavity}
\E[\varphi_K(S_n)] \leq \E[\varphi_K(S_{n-1})].
\end{equation}
\end{enumerate}
\end{theorem}

\begin{proof}
{(a)} By the additive property of Gaussian vectors, $S_{n,k} = \sum_{i=1}^n Y_{i,k} \sim N(0, n\Sigma_{kk})$, and more generally $(S_{n,1},\ldots,S_{n,K}) \sim N(\mathbf{0}, n\Sigma)$. Therefore $S_n \stackrel{d}{=} \sqrt{n}\,Z$ where $Z \sim N(\mathbf{0}, \Sigma)$, and $M_n = \max_k S_{n,k} \stackrel{d}{=} \sqrt{n}\max_k Z_k$. Taking expectations gives \eqref{eq:gaussian_Mn}.

{(b)} By part (a):
\begin{align*}
\E[\varphi_K(S_{i-1})] &= \E[M_i] - \E[M_{i-1}]\\
&= \E[\max_k Z_k]\bigl(\sqrt{i} - \sqrt{i-1}\bigr).
\end{align*}

{(c)} Follows directly from (b).

{(d)} The function $t \mapsto \sqrt{t}$ is concave, so $\sqrt{i+1} - \sqrt{i} \leq \sqrt{i} - \sqrt{i-1}$, which gives $\E[\varphi_K(S_i)] \leq \E[\varphi_K(S_{i-1})]$ by part (b).
\end{proof}

In \ref{eq:psi_gaussian}, we used the normalized selection premium function. The definition is as follows.

\begin{remark}[Universality of the normalized decay]
The formula $\psi_K(i) = \sqrt{i} - \sqrt{i-1}$ depends on neither the number of arms $K$ nor the covariance structure $\Sigma$. It is a purely geometric consequence of the $\sqrt{n}$-scaling of Gaussian random walks. In particular:
\begin{align*}
\psi_K(2) &= \sqrt{2} - 1 \approx 0.414, \qquad &\psi_K(5) &\approx 0.236,\\
\psi_K(10) &\approx 0.162, &\psi_K(26) &\approx 0.099.
\end{align*}
The $\alpha$-decay time $\tau_\alpha = \min\{i : \psi_K(i) < \alpha\}$ satisfies $\tau_\alpha \approx 1/(4\alpha^2)$ for small $\alpha$. In particular, $\tau_{0.1} = 26$ exactly (the smallest $i$ with $\sqrt{i} - \sqrt{i-1} < 0.1$).
\end{remark}

\subsection{Asymptotic decay for general distributions}

\begin{theorem}[Asymptotic growth---finite variance]\label{thm:general_qualitative}
Assume \textnormal{(A1)--(A3)} with $\sigma_k^2 = \Var(Y_{1,k}) < \infty$ for each $k$.
Let $\Sigma = \Cov(Y_1)$ and $c = \E[\max_k Z_k]$ where $Z \sim N(\mathbf{0}, \Sigma)$.
Then:
\begin{enumerate}[label=(\alph*)]
\item {(Square-root growth.)}
$\E[M_n] = c\sqrt{n} + o(\sqrt{n}).$

\item {(Bias concentration.)}
For any $0 < \alpha < 1$,
\begin{equation}\label{eq:bias_concentration}
\frac{\sum_{i=1}^{\lfloor \alpha n \rfloor} \E[\varphi_K(S_{i-1})]}{\E[M_n]} \to \sqrt{\alpha}
\quad \text{as } n \to \infty.
\end{equation}
That is, the first $\alpha$-fraction of steps generates a $\sqrt{\alpha}$-fraction of total bias.
\end{enumerate}
\end{theorem}

\begin{proof}
\textbf{(a)} By the multivariate CLT, $S_n/\sqrt{n} \Rightarrow Z \sim N(\mathbf{0}, \Sigma)$.
Since $h(x) = \max_k x_k$ is Lipschitz, the continuous mapping theorem gives
$M_n/\sqrt{n} \Rightarrow \max_k Z_k$.
For convergence of expectations, we verify uniform integrability:
\[
\E\bigl[(M_n/\sqrt{n})^2\bigr]
\leq \E\Bigl[\bigl(\textstyle\sum_k |S_{n,k}|/\sqrt{n}\bigr)^2\Bigr]
\leq K \sum_k \sigma_k^2 = K\,\mathrm{tr}(\Sigma) < \infty
\]
uniformly in $n$. Hence $\E[M_n/\sqrt{n}] \to c$.

\textbf{(b)} By part (a) applied at $\lfloor \alpha n \rfloor$:
\[
\sum_{i=1}^{\lfloor \alpha n \rfloor} \E[\varphi_K(S_{i-1})]
= \E[M_{\lfloor \alpha n \rfloor}]
= c\sqrt{\alpha n} + o(\sqrt{n}).
\]
Dividing by $\E[M_n] = c\sqrt{n} + o(\sqrt{n})$ yields the result.
\end{proof}

Equation~\eqref{eq:bias_concentration} makes precise the claim that
selection bias is a small-sample phenomenon.
For example, with $\alpha = 0.01$ (the first $1\%$ of steps),
the fraction of accumulated bias is approximately $\sqrt{0.01} = 0.1$: ten percent
of the total bias is generated in the first percent of the evaluation.
This concentration is a universal consequence of the $\sqrt{n}$-growth
and holds for {any} increment distribution with finite variance.

Monte Carlo computation of $\psi_K(i)$ for $i = 1, \ldots, 200$ across multiple configurations reveals the following (Figure~\ref{fig:premium_decay}):

\begin{figure}[h!]
\centering
\includegraphics[width=\textwidth]{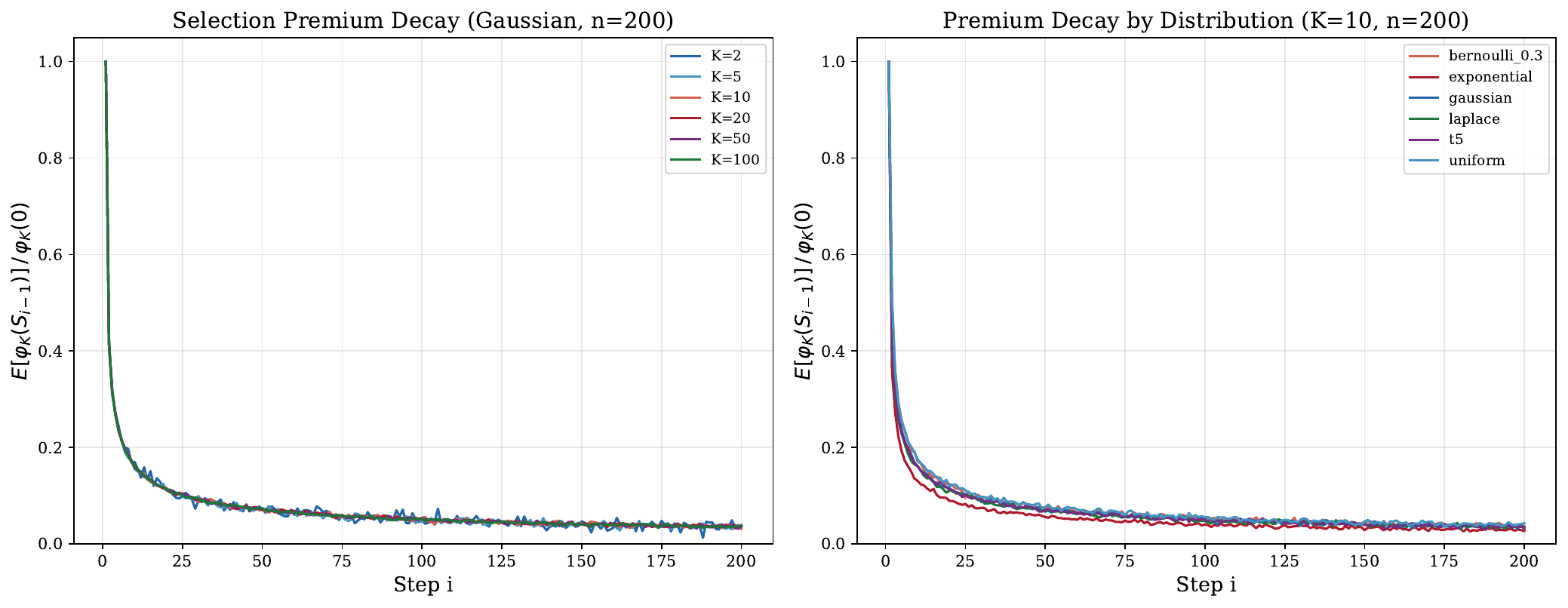}
\caption{{Selection premium decay.} Left: Normalized premium $\psi_K(i)$ for Gaussian arms with various $K$. Right: Distribution comparison for $K = 10$. The normalized curves are nearly identical across distributions, suggesting a universal decay law.}
\label{fig:premium_decay}
\end{figure}

\begin{enumerate}[label=(\roman*)]
\item {Rapid initial decay.} For $K = 10$ Gaussian arms ($\sigma = 1$), $\tau_{0.1} \approx 26$: the selection premium drops to $10\%$ of its initial value within approximately $26$ steps. 

\item {Approximately distribution-free shape.} When normalized by $\varphi_K(\mathbf{0})$, the decay curves are nearly identical across Gaussian, Student-$t(5)$, Exponential, Uniform, and Laplace distributions (Figure~\ref{fig:premium_decay}, right panel). This suggests a universal decay law that depends on $K$ but not on the increment distribution.

\end{enumerate}

The practical implication is as follows: selection bias accumulates primarily during approximately the first $\sqrt{n}$ steps, when models are still competing. Once a leader emerges, $\varphi_K(S_{i-1})$ becomes negligible and bias growth effectively stops. This is a ``competition--dominance transition.''

\subsection{Nonasymptotic bounds under sub-Gaussian increments}\label{sec:subgaussian}
The preceding results describe the qualitative and asymptotic behavior of the selection premium. We now complement these asymptotic statements with a {nonasymptotic} upper bound that holds uniformly in $n$ and $K$ under a sub-Gaussian moment condition.

The sub-Gaussian assumption allows us to control the moment generating function of $S_n$ and obtain a maximal inequality of order $\sqrt{n\log K}$. 
This bound recovers the $\sqrt{n}$ growth at fixed $K$ and the $\sqrt{\log K}$ dependence at fixed $n$, thereby unifying the time-scaling and model-scaling behaviors in a single inequality. 
\begin{definition}[sub-Gaussian]
    A random vector $Y_1\in\R^K$ is \emph{sub-Gaussian with proxy covariance} $\Sigma\succeq 0$ if for all $t\in\R^K$,
    \begin{equation}\label{eq:sg_def}
    \E\exp(\langle t, Y_1\rangle)\ \le\ \exp\Big(\frac{1}{2}\,t^\top \Sigma t\Big).
\end{equation}
\end{definition}

\begin{theorem}[Maximal selection bias under sub-Gaussian increments]\label{thm:subg_max}
Assume \textnormal{(A1)}--\textnormal{(A2)} and that $Y_1$ satisfies \eqref{eq:sg_def} for some $\Sigma\succeq 0$. Let $\lambda_{\max}(\Sigma)$ denote the largest eigenvalue of $\Sigma$. Then for all $n\ge1$,
\begin{equation}\label{eq:subg_max_bound}
\E[M_n]\ =\ \E\Big[\max_{1\le k\le K} S_{n,k}\Big]
\ \le\ \sqrt{2n\,\lambda_{\max}(\Sigma)\,\log K}.
\end{equation}
Consequently, the average per-step premium obeys
\[
\frac1n \sum_{i=1}^n \E[\varphi_K(S_{i-1})]
\le \sqrt{\frac{2\lambda_{\max}(\Sigma)\log K}{n}},
\]
and hence tends to $0$ at rate $n^{-1/2}$ for fixed $K$.
\end{theorem}

\begin{proof}
Fix $\theta>0$. By the inequality $\max_k x_k\le \theta^{-1}\log\sum_{k=1}^K e^{\theta x_k}$,
\[
\E[M_n]\le \frac{1}{\theta}\,\E\log\Big(\sum_{k=1}^K e^{\theta S_{n,k}}\Big)
\le \frac{1}{\theta}\,\log\E\Big(\sum_{k=1}^K e^{\theta S_{n,k}}\Big),
\]
where the last step is Jensen since $\log$ is concave.
Thus
\begin{equation}\label{eq:mgf_sum}
\E[M_n]\le \frac{1}{\theta}\log\Big(\sum_{k=1}^K \E e^{\theta S_{n,k}}\Big).
\end{equation}
Now $S_n=\sum_{i=1}^n Y_i$ is sub-Gaussian with proxy covariance $n\Sigma$ because the mgf bound tensorizes over independent sums:
for any $t\in\R^K$,
\[
\E e^{\langle t,S_n\rangle}
=\prod_{i=1}^n \E e^{\langle t,Y_i\rangle}
\le \exp\Big(\frac{n}{2}\,t^\top\Sigma t\Big).
\]
Taking $t=\theta e_k$ yields
\[
\E e^{\theta S_{n,k}}\le \exp\Big(\frac{n}{2}\,\theta^2\,e_k^\top\Sigma e_k\Big)
\le \exp\Big(\frac{n}{2}\,\theta^2\,\lambda_{\max}(\Sigma)\Big).
\]
Plugging into \eqref{eq:mgf_sum} gives
\[
\E[M_n]\le \frac{1}{\theta}\log\Big(K\exp\big(\tfrac{n}{2}\theta^2\lambda_{\max}(\Sigma)\big)\Big)
= \frac{\log K}{\theta}+\frac{n}{2}\theta\lambda_{\max}(\Sigma).
\]
Optimize over $\theta>0$ by choosing $\theta=\sqrt{\frac{2\log K}{n\lambda_{\max}(\Sigma)}}$, which yields \eqref{eq:subg_max_bound}. The remaining statements follow from the identity $\E[M_n]=\sum_{i=1}^n \E[\varphi_K(S_{i-1})]$ and dividing by $n$.
\end{proof}

\begin{figure}[h!]
\centering

\begin{subfigure}[t]{0.49\textwidth}
  \centering
  \includegraphics[width=\textwidth]{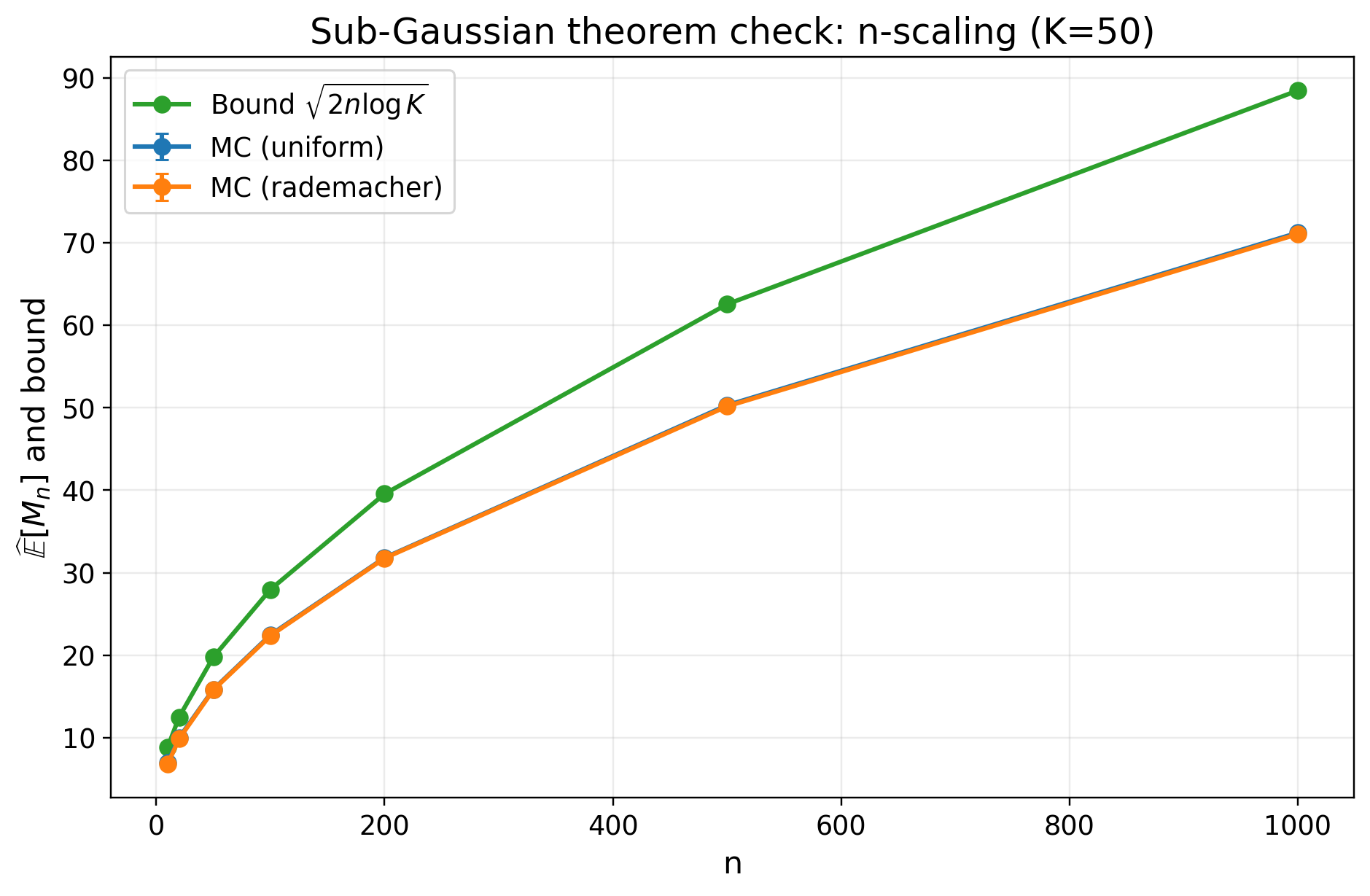}
  \caption{\(n\)-scaling at fixed \(K=50\).}
  \label{fig:subg_n_scaling_K50}
\end{subfigure}
\hfill
\begin{subfigure}[t]{0.49\textwidth}
  \centering
  \includegraphics[width=\textwidth]{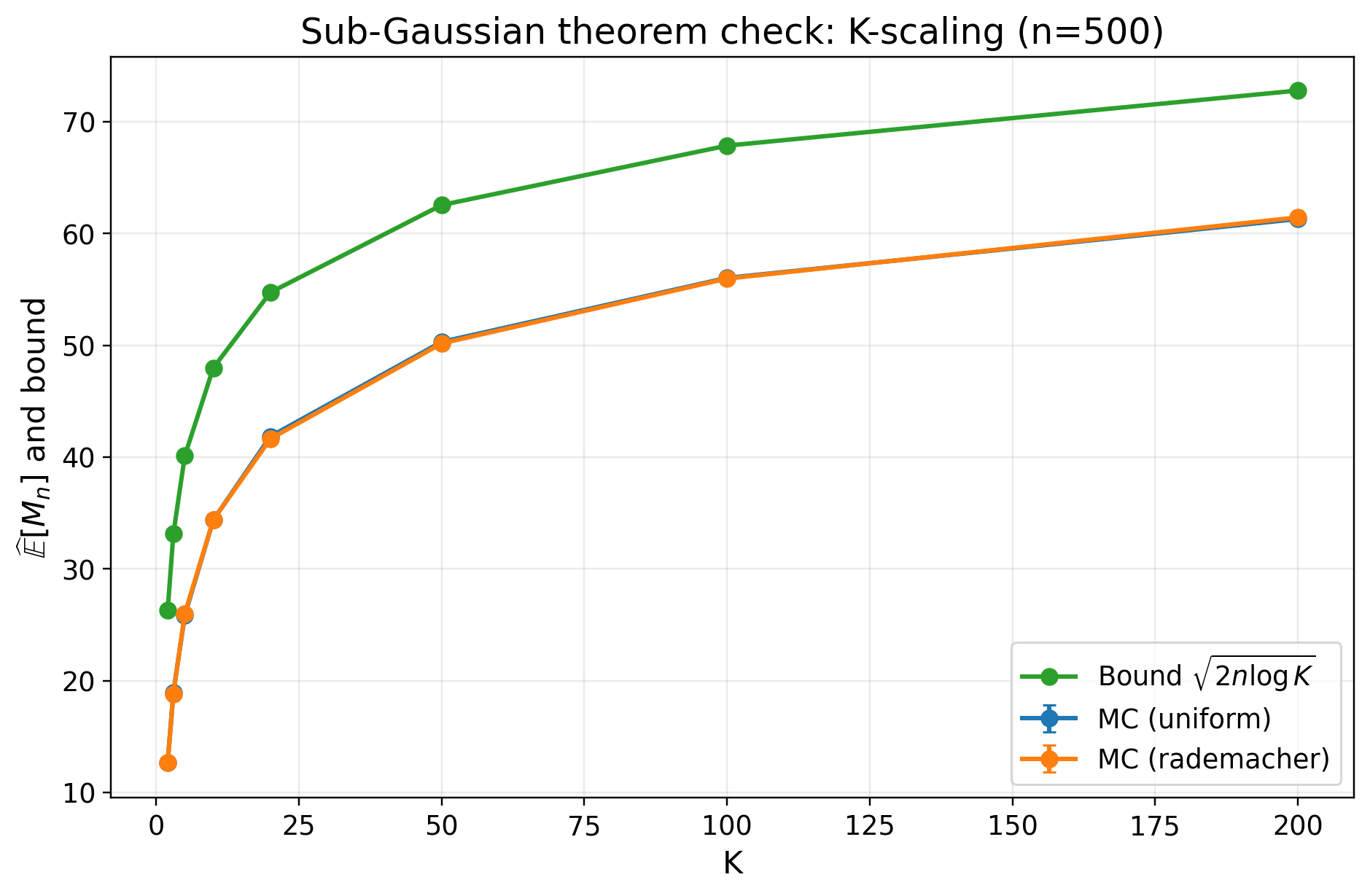}
  \caption{\(K\)-scaling at fixed \(n=500\).}
  \label{fig:subg_K_scaling_n500}
\end{subfigure}

\caption{Scaling in $n$ and $K$ for two uniform and Rademacher increment laws. Panel (a) verifies the $\sqrt{n}$ growth at fixed $K$; panel (b) verifies the $\sqrt{\log K}$ growth at fixed $n$.}

\label{fig:subg_scaling_side_by_side}
\end{figure}

Figure~\ref{fig:subg_scaling_side_by_side} provides an illustration of the $\sqrt{n\log K}$ envelope and its implied $\sqrt{n}$ (fixed $K$) and $\sqrt{\log K}$ (fixed $n$) scalings. We compare Monte Carlo estimates of $\E[M_n]$ for two i.i.d.\ centered, variance-one sub-Gaussian increment families:
(i) \(Y_{i,k}\sim \mathrm{Unif}[-\sqrt{3},\sqrt{3}]\) and (ii) Rademacher \(Y_{i,k}\in\{-1,+1\}\) with equal probability.
The two empirical curves nearly overlap. The green curve is the nonasymptotic bound from Theorem~\ref{thm:subg_max}, $\E[M_n]\le \sqrt{2n\,\lambda_{\max}(\Sigma)\,\log K}$.
In these simulations the coordinates are i.i.d.\ with unit sub-Gaussian proxy, so one may take the proxy covariance
\(\Sigma=I_K\), giving \(\lambda_{\max}(\Sigma)=1\) and the plotted bound \(\sqrt{2n\log K}\).

\section{Discussion}\label{sec:discussion}

The identity $\E[M_n] = \sum_{i=1}^n \E[\varphi_K(S_{i-1})]$
is not a deep result---it follows from conditioning and telescoping, but it identifies $\varphi_K$ as the right object of study.
The properties of $\varphi_K$---non-negativity, translation invariance,
maximum at ties, vanishing under dominance---encode the
competition--dominance structure of selection bias
in a single function. Also the extension to stopping times
(Theorem~\ref{thm:stopping})
is a genuine multi-arm generalization of Wald's equation,
requiring a non-trivial uniform integrability argument.

Several existing approaches address selection bias through different mechanisms.
Nested cross-validation separates selection from evaluation procedurally
but does not quantify the bias magnitude.
Tweedie's formula \citep{efron2011tweedie} considers the setting where a large number of estimates $z_i \sim N(\mu_i, \sigma^2)$ are observed, each with its own unknown mean $\mu_i$, and the largest $z_i$'s tend to overestimate their corresponding $\mu_i$'s. Tweedie's formula provides the correction $\hat\mu_i = z_i + \sigma^2\,(\log f)'(z_i)$, where $f$ is the marginal density. This shrinks individual point estimates toward the population center. Our decomposition addresses a different object: instead of correcting individual estimates, we characterize the aggregate selection bias $\E[\max_k S_k]$---the expected amount by which the winner's cumulative score exceeds its true mean. Tweedie's formula operates on the cross-section (many parallel estimates at a single time), while our decomposition operates along the time axis (a single set of $K$ models evaluated sequentially).

\citet{cawley2010over} showed that the variance of the model selection criterion---not just its bias---can cause overfitting during model selection, leading to optimistic performance estimates. Nested cross-validation addresses this by using an inner loop for model selection and an outer loop for performance evaluation, thereby separating selection from evaluation procedurally. This is an effective practical fix, and it avoids the bias rather than measuring it. $\E[\max_k S_{n,k}]$ is the bias that nested CV aims to eliminate, and the bias concentration law (Theorem~\ref{thm:general_qualitative}(b)) explains why even modest outer-loop sample sizes suffice.

Extensions to dependent increments (mixing conditions)
and correlated arms (shared architecture components)
require new tools, possibly from the decoupling framework
of \citet{de2012decoupling}.

\section{Conclusion}\label{sec:conclusion}

We have introduced the selection premium function $\varphi_K$
and developed it as a tool for the analysis of selection bias
when $K$ models are evaluated on shared data.
The per-step decomposition $\E[\max_k S_{n,k}] = \sum_{i=1}^n \E[\varphi_K(S_{i-1})]$,
while elementary in its derivation, extends to stopping times as a multi-arm Wald identity,
and establishes a universal bias concentration law \ref{sec:decay}.
For practitioners, the decomposition quantifies
how much optimistic bias to expect
and confirms that selection bias is predominantly
a small-sample phenomenon:
it accumulates during the early competition phase
and becomes negligible once a leader has emerged.

\section*{Acknowledgements}

We would like to thank Spencer Frei for helpful discussions. This research was supported by Google DeepMind under project number GT009019.

\bibliographystyle{plainnat}

\bibliography{refs} 


\appendix

Source code and all data files are available upon request on GitHub.
\end{document}